\documentclass[12pt]{article}
\usepackage{fullpage}
\usepackage{amssymb,amsmath,amsfonts,amsthm,latexsym,enumerate,url,cases}
\numberwithin{equation}{section}
\usepackage{hyperref}
\hypersetup{colorlinks=true,citecolor=blue,linkcolor=blue,urlcolor=blue}

\newtheorem{theorem}{Theorem}[section] %
\newtheorem{lemma}[theorem]{Lemma} %
\newtheorem{corollary}[theorem]{Corollary} %

\begin{document}
\title{On the denominators of harmonic numbers \footnote{This work was
supported by the National Natural Science Foundation of China (No.
11771211) and a project funded by the Priority Academic Program
Development of Jiangsu Higher Education Institutions.} }

\author{  Bing-Ling Wu and  Yong-Gao Chen\footnote{Corresponding author, E-mail: 390712592@qq.com (B.-L. Wu),
 ygchen@njnu.edu.cn(Y.-G. Chen) } \\
\small  School of Mathematical Sciences and Institute of Mathematics, \\
\small  Nanjing Normal University,  Nanjing  210023,  P. R. China
}
\date{}
\maketitle \baselineskip 18pt \maketitle \baselineskip 18pt

{\bf Abstract.} Let $H_n$ be the $n$-th harmonic number and let
$v_n$ be its denominator.  It is well known that $v_n$ is even for
every integer $n\ge 2$. In this paper, we study the properties of
$v_n$. One of our results is:  the set of positive integers $n$
such that $v_n$ is divisible by the least common multiple of $1,
2, \cdots, \lfloor {n^{1/4}}\rfloor $ has density one. In
particular, for any positive integer $m$, the set of positive
integers $n$ such that $v_n$ is divisible by $m$ has density one.



\vskip 5mm

\section{Introduction}

For any positive integer $n$, let
$$ H_n=1+\frac 12 +\frac 13 +\cdots +\frac 1n=\frac{u_n}{v_n}, \quad (u_n, v_n)=1, \ v_n>0.$$
The number $H_n$ is called \emph{$n$-th harmonic number}. In 1991,
Eswarathasan and Levine \cite{Eswarathasan} introduced $I_p$ and
$J_p$. For any prime number $p$, let $J_p$ be the set of positive
 integers
$n$ such that $p\mid u_n$ and let $I_p$ be the set of positive
 integers
$n$ such that $p\nmid v_n$. Here $I_p$ and $J_p$ are slightly
different from those in \cite{Eswarathasan}. In
\cite{Eswarathasan}, Eswarathasan and Levine considered $0\in I_p$
and $0\in J_p$. It is clear that $J_p\subseteq I_p$.

In 1991, Eswarathasan and Levine \cite{Eswarathasan}  conjectured
that $J_p$ is finite for any prime number $p$. In 1994, Boyd
\cite{Boyd} confirmed that $J_p$ is finite for $p\le 547$ except
$83, 127, 397$. For any set $S$ of positive integers, let
$S(x)=|S\cap [1, x]|$. In 2016, Sanna \cite{Sanna} proved that
$$J_p (x)\le 129 p^{\frac 23}x^{0.765}.$$  Recently, Wu and Chen
\cite{Wu} proved that \begin{equation}\label{eqn1}J_p (x)\le
3x^{\frac 23+\frac 1{25\log p}}.\end{equation}

For any positive integer $m$, let $I_m$ be the set of positive
 integers $n$ such that $m\nmid v_n$. In this paper, the following results are proved.

\begin{theorem}\label{thm2} The set of positive integers $n$
such that $v_n$ is divisible by the least common multiple of $1,
2, \cdots, \lfloor {n^{1/4}}\rfloor $ has density one.
\end{theorem}

\begin{theorem}\label{thm1} For any positive integer $m$ and any positive real number $x$, we have
$$I_m(x)\le 4 m^{\frac 1
3}x^{{\frac{2}{3}}+\frac{1}{25\log q_m}},$$ where $q_m$ is the
least prime factor of $m$.
\end{theorem}

From Theorem \ref{thm2} or Theorem \ref{thm1}, we immediately have
the following corollary.

\begin{corollary}\label{cor1} For any positive integer $m$, the set of
positive integers $n$ such that $m\mid v_n$ has density one.
\end{corollary}

\section{Proofs}

We always use $p$ to denote a prime. Firstly, we give the
following two lemmas.

\begin{lemma}\label{lem2a} For any prime $p$ and any positive integer $k$, we have
$$I_{p^k}=\{p^k{n_1}+r : n_1\in J_p\cup \{ 0\} ,\  0\le r\le p^k-1 \} \setminus \{ 0\} .$$
\end{lemma}

\begin{proof} For any integer $a$, let $\nu_p (a)$ be the $p$-adic
valuation of $a$. For any rational number $\alpha =\frac ab$, let
$\nu_p (\alpha )= \nu_p (a) -\nu_p (b)$. It is clear that $n\in
I_{p^k}$ if and only if $\nu_p (H_n)>-k$.

If $n<p^k$, then $\nu_p (H_n)\ge -\nu_p([1, 2, \cdots, n])>-k$. So
$n\in I_{p^k}$. In the following, we assume that $n\ge p^k$. Let
$$n=p^k n_1 +r,\quad  0\le r\le p^k-1,\ n_1, r\in \mathbb{Z}.$$
Then $n_1\ge 1$. Write \begin{equation}\label{eqn2}H_n=\sum_{m=1,
p^k\nmid m}^n \frac 1m +\frac 1{p^k} H_{n_1}=\frac b{p^{k-1}a}
+\frac
{u_{n_1}}{p^kv_{n_1}}=\frac{pbv_{n_1}+au_{n_1}}{p^kav_{n_1}},\end{equation}
where $p\nmid a$ and $(u_{n_1}, v_{n_1})=1$.

If ${n_1}\in J_p$, then $p\mid u_{n_1}$ and $p\nmid v_{n_1}$. Thus
$p\mid au_{n_1}+pbv_{n_1}$ and $\nu_p (p^kav_{n_1})=k$. By
\eqref{eqn2}, $\nu_p (H_n)>-k $.  So $n\in I_{p^k}$.

If ${n_1}\notin J_p$, then $p\nmid u_{n_1}$. Thus $p\nmid
au_{n_1}+pbv_{n_1}$. It follows from \eqref{eqn2} that $\nu_p
(H_n)\le -k$.  So $n\notin I_{p^k}$.

Now we have proved that $n\in I_{p^k}$ if and only if ${n_1}\in
J_p\cup \{ 0\}$.

This completes the proof of Lemma \ref{lem2a}.
\end{proof}

\begin{lemma}\label{lem2} For any prime power $p^k$ and any positive number $x$, we have
$$I_{p^k}(x)\le
4(p^k)^{{\frac{1}{3}}-\frac{1}{25\log
p}}x^{{\frac{2}{3}}+\frac{1}{25\log p}}.$$
\end{lemma}

\begin{proof}

If $x\le p^k$, then
$$I_{p^k}(x)\le x< 4x^{{\frac{1}{3}}-\frac{1}{25\log
p}}x ^{{\frac{2}{3}}+\frac{1}{25\log p}} \le
4(p^k)^{{\frac{1}{3}}-\frac{1}{25\log
p}}x^{{\frac{2}{3}}+\frac{1}{25\log p}}.$$ Now we assume that
$x>p^k$. By  Lemma \ref{lem2a} and \eqref{eqn1}, we have
\begin{eqnarray*}
I_{p^k}(x)&=&|\{p^k{n_1}+r\le x: n_1\in J_p\cup \{ 0\} ,\ 0\le
r\le p^k-1 \}| -1\\&\le& p^k\left(
{J_p(\frac{x}{p^k})}+1\right)\le
4(p^k)^{{\frac{1}{3}}-\frac{1}{25\log
p}}x^{{\frac{2}{3}}+\frac{1}{25\log p}}..
\end{eqnarray*}

This completes the proof of Lemma \ref{lem2}.
\end{proof}

\begin{proof}[Proof of Theorem \ref{thm2}]
Let $m_n$ be the least common multiple of $1, 2, \cdots, \lfloor
{n^{\theta}}\rfloor $ and let $\lfloor {n^{\theta}}\rfloor$ denote
the greatest integer not exceeding the real number $n^{\theta}$,
where $0<\theta <1 $ which will be given later. Let $T=\{n:
{m_n}\nmid{v_n}\}$. For any prime $p$ and any positive number $x$
with $p\le x^{\theta}$, let $\alpha_p$ be the integer such that
$p^{\alpha_p}\le x^{\theta}<p^{\alpha_p+1}$.

By the definition of $m_n$ and $T$,
$$T(x)\le\sum_{p\le x^{\theta}}I_{p^{\alpha_p}}(x).$$
Thus, by  Lemma \ref{lem2}, we have
$$
\sum_{p\le x^{\theta}}I_{p^{\alpha_p}}(x)\le 4\sum_{p\le
x^{\theta}}(p^{\alpha_p})^{\frac{1}{3}}x^{{\frac{2}{3}}+\frac{1}{25\log
p}} :=I_1+I_2,$$ where $$ I_1= 4\sum_{x^{\delta}<p\le
x^{\theta}}(p^{\alpha_p})^{\frac{1}{3}}x^{{\frac{2}{3}}+\frac{1}{25\log
p}},\quad I_2= 4\sum_{p\le
x^{\delta}}(p^{\alpha_p})^{\frac{1}{3}}x^{{\frac{2}{3}}+\frac{1}{25\log
p}} $$ and $\delta$ is a positive constant less than $\theta$
which will be given later.

If $p>x^{\delta}$, then
\begin{eqnarray*}
x^{\frac{1}{25\log p}}=e^{\frac{\log x}{25\log p}}\le
e^{\frac{\log x}{25{\delta}\log x}}=e^{\frac{1}{25{\delta}}}.
\end{eqnarray*}
It follows from $p^{\alpha_p}\le x^{\theta}$  that
\begin{eqnarray*}
I_1=4\sum_{x^{\delta}<p\le
x^{\theta}}(p^{\alpha_p})^{\frac{1}{3}}x^{{\frac{2}{3}}+\frac{1}{25\log
p}}  \le  4 e^{\frac{1}{25\delta}}\sum_{x^{\delta}<p\le
x^{\theta}} x^{\frac{\theta}{3}+\frac{2}{3}}\ll \frac{1}{\log x}
x^{\frac{4\theta}{3}+\frac{2}{3}}.
\end{eqnarray*}

For $I_2$, by $p^{\alpha_p}\le x^{\theta}$  we have
\begin{eqnarray*}
I_2 &=& 4\sum_{p\le x^{\delta}} (p^{\alpha_p})^{\frac{1}{3}}
x^{\frac{2}{3}+\frac{1}{25\log p}}\\
&\le&4\sum_{p\le x^{\delta}}
x^{\frac{\theta}{3}+\frac{2}{3}+\frac{1}{25\log 2}}\\
&\ll & \frac{1}{ \log
x}x^{\delta+{\frac{\theta}{3}}+{{\frac{2}{3}}+\frac{1}{25\log
2}}}.
\end{eqnarray*}

We choose  $\theta=\frac 1 4$ and $\delta=0.1$. Then
$$I_1\ll \frac x{\log x} ,\quad I_2\ll x^{0.91}.$$
Therefore, $$T(x)\le \sum_{p\le
x^{\theta}}I_{p^{\alpha_p}}(x)=I_1+I_2\ll \frac x{\log x}.$$ It
follows that the set of positive integers $n$ such that $v_n$ is
divisible by the least common multiple of $1, 2, \cdots, \lfloor
{n^{1/4}}\rfloor $ has density one. This completes the proof of
Theorem \ref{thm2}.\end{proof}

\begin{proof}[Proof of Theorem \ref{thm1}]
We use induction on $m$ to prove Theorem \ref{thm1}.

By Lemma \ref{lem2}, Theorem \ref{thm1} is true for $m=2$. Suppose
that Theorem \ref{thm1} is true for all integers less than $m$
$(m>2)$.

If $x\le m$, then
$$I_m(x)\le x<4 x^{\frac 1
3}x^{{\frac{2}{3}}+\frac{1}{25\log {q_m}}}\le 4 m^{\frac 1
3}x^{{\frac{2}{3}}+\frac{1}{25\log {q_m}}}.$$

In the following, we always assume that $x>m$.

If $m$ is a prime power, then, by Lemma \ref{lem2}, Theorem
\ref{thm1} is true.  Now we assume that $m$ is not a prime power.
Let $p^\alpha $ be the least prime power divisor of $m$ such that
$m=p^\alpha m_1$ with $p\nmid m_1$. Then $m_1>p^\alpha $. It is
clear that $I_m=I_{m_1}\bigcup ( I_{p^\alpha } \setminus
I_{m_1})$. By Lemma \ref{lem2a} and the definition of $p^\alpha $,
$\{ 1,2, \cdots , p^\alpha -1\} \subseteq I_{m_1} $. Hence
$$I_m(x)=I_{m_1}(x)+( I_{p^\alpha } \setminus I_{m_1})(x) \le
I_{m_1}(x)+ I_{p^\alpha }(x)-(p^\alpha -1).$$ By the inductive
hypothesis, we have $$I_{m_1}(x) \le 4m_1^{\frac{1}{3}}
x^{{\frac{2}{3}}+\frac{1}{25\log q_{m_1}}}\le 4m_1^{\frac{1}{3}}
x^{{\frac{2}{3}}+\frac{1}{25\log q_m}}.$$ It follows that
\begin{equation}\label{eq1a0}I_m(x)\le 4 m_1^{\frac{1}{3}} x^{{\frac{2}{3}}+\frac{1}{25\log
q_m}}+ I_{p^\alpha } (x) -(p^\alpha -1).\end{equation}

We divide into the following three cases:

{\bf Case 1:} $p^\alpha\ge 8$. Then $m_1>p^\alpha \ge 8$. By Lemma
\ref{lem2}, we have
\begin{equation*}I_{p^\alpha}(x) \le 4(p^\alpha)^{\frac{1}{3}} x^{{\frac{2}{3}}+\frac{1}{25\log
q_m}}.\end{equation*} It follows from \eqref{eq1a0} that
\begin{eqnarray*}I_m(x)
&\le &4 m_1^{\frac{1}{3}} x^{{\frac{2}{3}}+\frac{1}{25\log
q_m}} + 4 (p^\alpha)^{\frac{1}{3}}
x^{{\frac{2}{3}}+\frac{1}{25\log q_m}}\\
&=& 4 \left( \frac 1{(p^\alpha)^{\frac{1}{3}}} +\frac
1{m_1^{\frac{1}{3}}} \right)  m^{\frac{1}{3}}
x^{{\frac{2}{3}}+\frac{1}{25\log q_m}}\\
&\le & 4 m^{\frac{1}{3}} x^{{\frac{2}{3}}+\frac{1}{25\log q_m}}.
\end{eqnarray*}

{\bf Case 2:} $p^\alpha<8$, $p=2$. Then  $p^\alpha=2$ or $4$ and
$x>m\ge 2\times 3= 6$. By Lemma \ref{lem2a} and $J_2=\emptyset$,
we have $I_4=\{ 1, 2, 3 \} $ and $I_2=\{ 1 \} $. It is clear that
$ I_{p^\alpha } (x) -(p^\alpha -1)=0$.
 It
follows from \eqref{eq1a0} that
\begin{eqnarray*}I_m(x)\le   4 m_1^{\frac{1}{3}} x^{{\frac{2}{3}}+\frac{1}{25\log
q_m}} < 4  m^{\frac{1}{3}} x^{{\frac{2}{3}}+\frac{1}{25\log q_m}}.
\end{eqnarray*}

{\bf Case 3:} $p^\alpha<8$, $p\neq 2$. Then  $\alpha=1$ and $p=3$,
$5$ or $7$. In addition, $x>m\ge 3\times 4= 12$. Noting that
$m^{\frac 1 3}-{m_1}^{\frac 1 3}={m_1}^{\frac 1 3}(p^{\frac
{\alpha}{3}}-1)\ge 4^{\frac 1 3}(3^{\frac 1 3}-1)>\frac 12$, by
\eqref{eq1a0}, it is enough to prove that $I_{p} (x) -(p -1)\le
2x^{\frac 2 3}$. By Lemma \ref{lem2a}, we have
$$I_{p}=\{p{n_1}+r : n_1\in J_p\cup \{ 0\} ,\  0\le r\le p-1 \} \setminus \{ 0\} .$$
 By
\cite{Eswarathasan}, $ J_3=\{2, 7, 22\}$, $ J_5=\{4, 20, 24\}$ and
$$J_7=\{6, 42, 48, 295, 299, 337, 341, 2096, 2390, 14675, 16731,
16735, 102728\} .$$ If $x\ge 7^3$, then $I_{p} (x) -(p -1)\le 91
\le 2x^{\frac 2 3}$. If $35<x< 7^3$, then $I_{p} (x) -(p -1)\le
21\le 2x^{\frac 2 3}$. If $12<x\le 35$, then $I_{p} (x) -(p -1)\le
6\le 2x^{\frac 2 3}$.

This completes the proof of  Theorem \ref{thm1}. \end{proof}

\renewcommand{\refname}{References}

\end{document}